\documentclass[a4paper,11pt]{amsart}

\usepackage{amsmath,amsfonts,amsthm,amssymb,ascmac,amscd}
\usepackage{bm}
\usepackage{graphicx}
\usepackage{enumerate}
\usepackage{overpic}
\usepackage{mathtools}
\usepackage{hyperref}
\usepackage{xcolor}
\hypersetup{
bookmarksnumbered=true, 
colorlinks=true, 
citecolor=black,
linkcolor=black,
urlcolor=black,
}
\numberwithin{equation}{section}
\usepackage{cleveref}
\crefname{definition}{Definition}{Definitions}
\Crefname{definition}{Definition}{Definitions}
\crefname{theorem}{Theorem}{Theorems}
\Crefname{theorem}{Theorem}{Theorems}
\crefname{lemma}{Lemma}{Lemmas}
\Crefname{lemma}{Lemma}{Lemmas}
\crefname{corollary}{Corollary}{Corollaries}
\Crefname{corollary}{Corollary}{Corollaries}
\crefname{proposition}{Proposition}{Propositions}
\Crefname{proposition}{Proposition}{Propositions}
\crefname{example}{Example}{Examples}
\Crefname{example}{Example}{Examples}
\crefname{remark}{Remark}{Remarks}
\Crefname{remark}{Remark}{Remarks}
\crefname{question}{Question}{Questions}
\Crefname{question}{Question}{Questions}
\crefname{equation}{}{}
\Crefname{equation}{}{}
\crefname{figure}{Figure}{Figures}
\Crefname{figure}{Figure}{Figures}
\theoremstyle{definition}
\newtheorem{theorem}{Theorem}[section]
\newtheorem*{theorem*}{Theorem}
\newtheorem{definition}[theorem]{Definition}
\newtheorem*{definition*}{Definition}
\newtheorem{proposition}[theorem]{Proposition}
\newtheorem*{proposition*}{Proposition}

\newtheorem*{example*}{Example}
\newtheorem{remark}[theorem]{Remark}
\newtheorem*{remark*}{Remark}

\newtheorem*{recall*}{Recall}
\newtheorem{lemma}[theorem]{Lemma}
\newtheorem*{lemma*}{Lemma}
\newtheorem{corollary}[theorem]{Corollary}
\newtheorem*{corollary*}{Corollary}

\newtheorem*{question*}{Question}

\newtheorem*{conjecture*}{Conjecture}

\newtheorem*{exercise*}{Exercise}
\newtheorem{claim}[theorem]{Claim}
\newtheorem*{claim*}{Claim}

\newtheorem*{fact*}{Fact}
\newtheorem{theorema}{Theorem}
\newtheorem*{theorema*}{Theorem}

\newtheorem{corollarya}[theorema]{Corollary}
\newtheorem*{corollarya*}{Theorem}
\newcommand{\Z}{\mathbb{Z}}

\DeclareMathOperator{\id}{id}

\newcommand{\Teich}{\mathcal{T}}
\newcommand{\elTeich}{\hat{\mathcal{T}}}
\DeclareMathOperator{\Ext}{Ext}
\DeclareMathOperator{\Area}{Area}
\DeclareMathOperator{\Mod}{Mod}
\newcommand{\Thin}{\mathrm{Thin}}
\newcommand{\multicurve}[1]{\mathcal{C}^{[#1]}}
\newcommand{\interpolate}[1]{\mathcal{P}^{[#1]}}

\allowdisplaybreaks[4]
\title[Electric Teichmüller spaces and $k$-multicurve graphs]{Electric Teichmüller spaces and $k$-multicurve graphs}
\author{Kento Sakai}
\address{Graduate School of Mathematical Sciences, The University of Tokyo, 3-8-1, Komaba, Meguro-ku, Tokyo, 153-8914, Japan}
\email{kento@ms.u-tokyo.ac.jp}
\date{\today}
\begin{document}

\begin{abstract}
    Masur and Minsky showed that the curve graph is quasi-isometric to the Teichmüller space electrified  along its thin part, and hence the Teichmüller space is weakly relatively hyperbolic with respect to the thin part. 
    In this paper, we extend this result to the $k$-multicurve graph by electrifying the Teichmüller space along the thin part where the extremal length of $k$ curves is sufficiently small. 
    A key ingredient is a bound on the $k$-multicurve graph distance in terms of the intersection number, which is obtained by adapting the upper bound for the pants graph due to Lackenby and Yazdi.
\end{abstract}
\maketitle

\section{Introduction}

The \textit{Teichmüller space} of a surface is the space of conformal (equivalently, hyperbolic) structures on the surface, up to isotopy.
In low-dimensional topology, the mapping class group is a fundamental object, and it is often studied in close connection with Teichmüller space.
To describe the large-scale and asymptotic geometry of Teichmüller space, a number of quasi-isometric combinatorial models have been proposed.
The \textit{curve graph}, introduced by Harvey \cite{harvey1981boundary}, has been a topic of interest in low-dimensional topology. 
The vertices of the curve graph are isotopy classes of essential simple closed curves on the surface, and two vertices are joined by an edge if the corresponding curves can be realized disjointly.
Masur and Minsky \cite{masur1999complex} showed that the curve graph is a hyperbolic metric space.

Masur and Minsky also proved that the curve graph is quasi-isometric to the metric space obtained by electrifying (or coning off) Teichmüller space along its thin part \cite[Theorem~1.2]{masur1999complex}.
Here the thin part consists of those Riemann surfaces on which the extremal length of some essential simple closed curve is at most a fixed constant $\epsilon_0>0$.
Each component of the thin part is quasi-isometric to a product of metric spaces \cite{minsky1996product}, and hence provides an obstruction to the Teichmüller space being hyperbolic.
In particular, Teichmüller space is weakly relatively hyperbolic with respect to the thin part.

In this paper, we extend \cite[Theorem~1.2]{masur1999complex} to the \emph{$k$-multicurve graph}, introduced in \cite{erlandsson2017multicurve}.
Let $\Sigma$ be a surface of genus $g$ with $n$ punctures.
The vertices of the $k$-multicurve graph $\multicurve{k}(\Sigma)$ are $k$-multicurves on $\Sigma$; that is, multicurves consisting of exactly $k$ pairwise disjoint curves, and two distinct vertices are joined by an edge if the corresponding multicurves contain the same $(k-1)$-multicurve and the remaining two complementary curves intersect  minimally (see \cref{definition:k-multicurve} for details).
For a $k$-multicurve $\alpha$, let $\Thin_\alpha$ be the subset of the Teichmüller space $\Teich(\Sigma)$ consisting of Riemann surfaces on which each component of $\alpha$ has the extremal length at most $\epsilon_0$.
Let $\elTeich^{k}(\Sigma)$ denote the metric space obtained by electrifying $\Teich(\Sigma)$ along the collection $\{\Thin_\alpha\}_{\alpha\in \multicurve{k}(\Sigma)}$.
The main purpose of this paper is to show the following.
\begin{theorema}\label{theoremA}
The electrified Teichmüller space $\elTeich^{k}(\Sigma)$ is quasi-isometric to the $k$-multicurve graph $\multicurve{k}(\Sigma)$.
\end{theorema}

Vokes introduced the notion of \emph{twist-freeness} for graphs of multicurves, and proved that any twist-free graph of multicurves is a hierarchically hyperbolic space \cite{vokes2022hirarchical}. 
Consequently, the general theory of hierarchically hyperbolic spaces \cite{behrstock2017hierarchyI,behrstock2019hierarchyII} yields many large-scale geometric properties of twist-free graphs of multicurves \cite{vokes2022hirarchical,russel2022thickness}. 
In particular, hyperbolicity and quasi-flat rank are determined by the number of pairwise disjoint \emph{witnesses} for the graph of multicurves, where a witness is a connected subsurface that essentially intersects every vertex of the graph of multicurves.
In \cite{kuno2026interpolating}, an explicit formula is given for the maximal number $m(g,n,k)$ of pairwise disjoint witnesses for $\multicurve{k}(\Sigma_{g,n})$:

\begin{equation*}
    m(g,n,k)=
        \begin{cases}
            \min \left\{\left\lfloor \frac{-\chi(\Sigma)}{a(\xi(\Sigma)+1-k)}\right\rfloor, \left\lfloor \frac{\xi(\Sigma)+1}{\xi(\Sigma)+2-k}\right\rfloor \right\} & (n,k)\neq (0,1) \\
            1 & (n,k)=(0,1),
        \end{cases}
\end{equation*}
where $a(x)= \lceil (2x+1)/3 \rceil$, $\chi(\Sigma)=2-2g-n$ and $\xi(\Sigma)=3g-3+n$.
As an immediate consequence, we obtain the following.

\begin{corollarya}
If $m(g,n,k)=1$, then $\Teich(\Sigma)$ is weakly relatively hyperbolic with respect to the collection $\{\Thin_\alpha\}_{\alpha\in\multicurve{k}(\Sigma)}$.
\end{corollarya}

\begin{corollarya}\label{corollaryC}
    Let $\elTeich^k(\Sigma_{g,n})$ denote the electrified Teichmüller space along the collection $\{\Thin_\alpha\}_{\alpha\in\multicurve{k}(\Sigma)}$.
    \begin{enumerate}
        \item \textbf{Hyperbolic case.} 
        $\elTeich^k(\Sigma_{g,n})$ is hyperbolic if and only if $m(g,n,k)=1$.
        \item \textbf{Relatively hyperbolic case.} $\elTeich^k(\Sigma_{g,n})$ is relatively hyperbolic if and only if 
        \begin{itemize}
            \item $g$ is even, $n$ is even and at least $2$, and $k=(3g+n)/2$,
            \item $g$ is even, $n=0$, and $k\in\{3g/2,(3g+2)/2\}$,
            \item $g$ is odd, $n\in\{0,2\}$, and $k=(3g+3)/2$, or 
            \item $g$ is odd, $n$ is odd and at least $3$, and $k=(3g+n)/2$.
        \end{itemize}
        \item \textbf{Thick case.} $\elTeich^k(\Sigma_{g,n})$ is thick if and only if neither (1) nor (2) holds.
    \end{enumerate} 
\end{corollarya}

\begin{corollarya}
    The quasi-flat rank of $\elTeich^k(\Sigma_{g,n})$ is equal to $m(g,n,k)$.
\end{corollarya}

To prove \Cref{theoremA}, we establish an upper bound on the distance between two vertices of $\multicurve{k}(\Sigma)$ in terms of the intersection number of the corresponding multicurves.
\begin{theorema}\label{theorema:E}
    For any $k$-multicurves $\alpha,\beta\in\multicurve{k}(\Sigma)$, we have 
    \begin{equation*}
        d_{\multicurve{k}}(\alpha,\beta)\leq 6\cdot 4^{6g-6+2n-2k}i(\alpha,\beta)^2+f(k),
    \end{equation*}
    where $f(k)=\min\{k,\xi(\Sigma)-k\}$.
\end{theorema}

Our proof of \Cref{theorema:E} uses the quadratic upper bound on pants graph distance in terms of intersection number due to Lackenby and Yazdi \cite{lackenby2024bounds}.
In the latter part of \cite{lackenby2024bounds}, they also obtain a sharper bound that grows only logarithmically with the intersection number.
By using this, one can also derive an upper bound of the $k$-multicurve graph distance in terms of a logarithmic function of the intersection number; however, the quadratic upper bound suffices to show \Cref{theoremA}, and so we adopt it for simplicity.

Finally, we briefly discuss related work.
Mahan Mj \cite{mj2009interpolating} constructed a graph interpolating between the curve graph and the pants graph.
The graph is called the \textit{complexity-$\xi$} graph, and we denote it by $\interpolate{\xi}(\Sigma)$.
He showed that the graph is quasi-isometric to the electrified Cayley graph $\Gamma(\Sigma,\xi)$ of the mapping class group $\Mod(\Sigma)$, where the electrifying is taken over left cosets of mapping class subgroups associated to subsurfaces of complexity at most $\xi$.
This extends \cite[Theorem~1.3]{masur1999complex}.
Mahan Mj also determined the quasi-flat ranks of both the complexity-$\xi$ graph and the corresponding electrified Cayley graph.
From this perspective, our results may be viewed as a Teichmüller-space analogue of Mahan Mj's work on mapping class groups.

Masur and Minsky \cite{masur2000complex} introduced the \emph{marking graph} $\mathcal{M}(\Sigma)$, which is quasi-isometric to the Cayley graph of the mapping class group $\Mod(\Sigma)$.
Durham \cite{durham2016augmented} defined the \emph{augmented marking graph} $\mathcal{AM}(\Sigma)$, which is quasi-isometric to Teichmüller space $\Teich(\Sigma)$ equipped with the Teichmüller metric.
Brock \cite{brock2003weil} proved that the pants graph $\mathcal{P}(\Sigma)$ is quasi-isometric to $\Teich(\Sigma)$ equipped with the Weil-Petersson metric.
We summarize these quasi-isometries in \Cref{table:QIMetricSpaces}. 
\begin{table}[h]
    \begin{center}
        \begin{tabular}{c c c} 
            \textbf{Multicurve graphs} & \textbf{MCG} & \textbf{Teichmüller space} \\ \hline
            $\mathcal{AM}(\Sigma)$ & \  & $\Teich(\Sigma)$ \\
            $\mathcal{M}(\Sigma)$  & $\Mod(\Sigma)$ & \  \\
            $\multicurve{\xi_0}(\Sigma)=\mathcal{P}(\Sigma)$  & $\Gamma(\Sigma,-1)$, $\Gamma(\Sigma,0)$ & $\Teich^{\mathrm{WP}}(\Sigma)$, $\elTeich^{\xi_0}(\Sigma)$ \\
            $\vdots$ & $\vdots$ & $\vdots$ \\
            $\multicurve{k}(\Sigma),\  \interpolate{\xi_0-k}(\Sigma)$ & $\Gamma(\Sigma,\xi_0-k)$
            & $\elTeich^{k}(\Sigma)$ \\
            $\vdots$ & $\vdots$ & $\vdots$ \\
            $\mathcal{C}(\Sigma)$, $\interpolate{\xi_0-1}(\Sigma)$ & $\Gamma(\Sigma,\,\xi_0-1)$ & $\elTeich^{1}(\Sigma)$ \\
        \end{tabular}
        \vspace{2mm}
    \end{center}
\caption{Here we set $\xi_0=3g-3+n$. The metric spaces in each row are mutually quasi-isometric.}
\label{table:QIMetricSpaces}
\end{table}

Fix nonnegative integers $g,n$ with $2g-2+n>0$, and let $k_0$ be the largest integer $k$ satisfying $m(g,n,k)=1$.
Then $k_0$ is approximately half of $3g-3+n$.
By \Cref{corollaryC}, every metric space appearing in \Cref{table:QIMetricSpaces} above the row corresponding to the $k_0$-multicurve graph is not hyperbolic.
On the other hand, the metric completion of Teichmüller space equipped with the Weil-Petersson metric is known to be a $\mathrm{CAT}(0)$ space \cite{yamada2004completion,wolpart2002completion}.
By contrast, Teichmüller space with the Teichmüller metric is not a $\mathrm{CAT}(0)$ space \cite{masur1975geodesic} (see also \cite{liu2020distance}).

\subsection*{Acknowledgements}
I would like to thank Mahan Mj.
This work is inspired by his answer for my question.
I am grateful to Nariya Kawazumi, Hidetoshi Masai, Ryo Matsuda, and Ken'ich Ohshika for their valuable comments and encouragement.
This work is supported by JSPS KAKENHI Grant Number	25KJ0069.

\section{Preliminaries}

\subsection{Metric geometry}
Let $(X,d_X)$ and $(Y,d_Y)$ be metric spaces.
A map $f\colon X\to Y$ is called a \emph{quasi-isometry} if the following conditions hold.
\begin{enumerate}
    \item There exist constants $K\geq 1, L>0$ such that, for any $x_1,x_2 \in X$, 
    \begin{equation*}
        \frac{1}{K} d_X(x_1,x_2)-L \leq d_Y(f(x_1),f(x_2)) \leq Kd_X(x_1,x_2)+L.
    \end{equation*}
    \item There exists $C>0$ such that every point of $Y$ lies within distance $C$ of $f(X)$.
\end{enumerate}
If a map $f\colon X\to Y$ satisfies the condition (1), it is called a \emph{quasi-isometric embedding}.
The \emph{quasi-flat rank} of a metric space $X$ is the largest integer $n$ for which there exists a quasi-isometric embedding $\Z^n\to X$.
\vspace{2mm}

We recall the definition of an electrified metric space, following \cite{farb1998relative}.
Let $X$ be a geodesic metric space, and let $\mathcal{H}=\{H_\alpha\}_{\alpha\in A}$ be a family of connected subsets of $X$.
For each $\alpha\in A$, we introduce a new point $v_\alpha$ and join every point of $H_\alpha$ to $v_\alpha$ with a segment of length $1/2$.
We denote the resulting space by $\hat{X}$, and equip $\hat{X}$ with the path metric $d_e$.
We refer to the construction of the metric space $(\hat{X},d_e)$ as \emph{electrifying} (or \emph{coning off}) $X$ along $\mathcal{H}$.

\subsection{Multicurve graphs}

Let $\Sigma = \Sigma_{g,n}$ be a surface of genus $g$ with $n$ punctures.
A simple closed curve on $\Sigma$ is said to be \emph{essential} if it is not homotopic to a point or to a puncture.
A \emph{$k$-multicurve} on $\Sigma$ is a set of isotopy classes of pairwise disjoint essential simple closed curves on $\Sigma$ whose cardinality is $k$, where $1 \le k \le 3g-3+n$.
In particular, a $(3g-3+n)$-multicurve on $\Sigma$ is called a \emph{pants decomposition} of $\Sigma$.

We define the $k$-multicurve graph $\multicurve{k}(\Sigma)$, which was introduced in \cite{erlandsson2017multicurve}.

\begin{definition}\label{definition:k-multicurve}
    The \emph{$k$-multicurve graph} $\multicurve{k}(\Sigma)$ is the graph whose vertices are $k$-multicurves on $\Sigma$.
    Two $k$-multicurves $\alpha$ and $\beta$ are joined by an edge if there exists a common $(k-1)$-multicurve $\nu \subset \alpha \cap \beta$ such that the two isotopy classes $\alpha \setminus \nu$ and $\beta \setminus \nu$ intersect minimally on $\Sigma \setminus \nu$.
    More precisely,
    \begin{itemize}
        \item if $k < 3g-3+n$, then $i(\alpha \setminus \nu, \beta \setminus \nu) = 0$;
        \item if $k = 3g-3+n$ and one of the connected components of
        $\Sigma \setminus \nu$ is a once-holed torus, then
        $i(\alpha \setminus \nu, \beta \setminus \nu) = 1$;
        \item if $k = 3g-3+n$ and one of the connected components of
        $\Sigma \setminus \nu$ is a four-holed sphere, then
        $i(\alpha \setminus \nu, \beta \setminus \nu) = 2$.
    \end{itemize}
\end{definition}

We endow the $k$-multicurve graph $\multicurve{k}(\Sigma)$ with the graph metric in which each edge has length $1$.
We denote the distance function of $\multicurve{k}(\Sigma)$ by $d_{\multicurve{k}}$.
Since the $(3g-3+n)$-multicurve graph is the same as the pants graph $\mathcal{P}(\Sigma)$,
the distance function of the $(3g-3+n)$-multicurve graph is often denoted by $d_{\mathcal{P}}$.

\subsection{Quadratic differentials and measured foliations}

We recall the construction of the measured foliation associated to a holomorphic quadratic differential (see, for example \cite{farb2012primer}).
Let $X$ be a Riemann surface homeomorphic to $\Sigma$, and
let  $q$ be a holomorphic quadratic differential on $X$ with finite $L^1$-norm.
Away from zeros of $q$, there is a holomorphic coordinate $\zeta=\xi+i\eta$ in which $q=d\zeta^2$.
Such a coordinate is uniquely determined up to $\zeta \mapsto \pm \zeta+c$, and hence the horizontal lines $\{\eta=\text{constant}\}$ define a foliation on $X$.
A zero of $q$ of order $m$ gives an $(m+2)$-pronged singularity of this foliation.
Moreover, the density $|d\eta|$ defines a transverse measure on the singular foliation of $q$, and the resulting measured foliation is called the \emph{horizontal measured foliation} of $q$ and denoted by $\mathcal{F}_h(q)$.
It is known that the map $q\mapsto \mathcal{F}_h(q)$ is a homeomorphism \cite{hubbard1979quadratic}.
Therefore, for each $k$-multicurve $\alpha$, we can choose a holomorphic quadratic differential $q_\alpha$ whose horizontal measured foliation corresponds to $\alpha$.

\subsection{Extremal length and Teichmüller space}
Let $\Sigma=\Sigma_{g,n}$  be a surface of genus $g$ with $n$ punctures.
We suppose that $2-2g-n<0$.
Let $X$ be a Riemann surface homeomorphic to $\Sigma$.
We define the \emph{extremal length} $\Ext_X(\alpha)$ of a multicurve $\alpha$ with respect to $X$ by
\begin{equation*}
    \Ext_X(\alpha) = \sup_{\rho} \frac{\ell_\rho(\alpha)^2}{\Area(\rho)},
\end{equation*}
where the supremum is taken over all conformal metrics $\rho$ on $X$.
It is well known that extremal length extends continuously to a function on the space of measured foliations on $\Sigma$.

The \textit{Teichmüller space} $\Teich(\Sigma)$ of $\Sigma$ is the space of isotopy classes of punctured Riemann surface structures on $\Sigma$.
Let $X, Y\in\Teich(\Sigma)$ be Riemann surfaces.
For each quasiconformal map $f\colon X\to Y$, the \textit{maximal dilatation $K(f)$} is defined by 
\begin{equation*}
    K(f)=\frac{1+\|\mu_f\|_\infty}{1-\|\mu_f\|_\infty},
\end{equation*}
where $\|\mu_f\|_\infty$ is the $L^\infty$-norm of the Beltrami differential $\mu_f$ of $f$.
Then we define a distance function $d_{\Teich}$ on $\Teich(\Sigma)$ as 
\begin{equation*}
    d_{\Teich}(X,Y)\coloneqq \frac12 \inf_{f\sim\id}\log K(f).
\end{equation*}
The distance $d_{\Teich}$ is called the \textit{Teichmüller distance} of $\Teich(\Sigma)$.
For any $X,Y\in \Teich(\Sigma)$, there exists a unique extremal map for maximal dilatation in the isotopy class of the identity map on $\Sigma$, which is called the \textit{Teichmüller map}.
Let $f \colon X \to Y$ be the Teichmüller map with dilatation $K$.
Then there exists a unique holomorphic quadratic differential $q$ on $X$ such that $f$ is represented by the map $\xi+i\eta \mapsto K\xi+i\eta$ in the canonical coordinates $\zeta = \xi+i\eta$ associated with $q$.
Then we denote $Y$ by $X_{q,K}$.
The one-parameter family $\{X_{q,e^{2t}}\}_{t \ge 0}$ of Riemann surfaces is called the \textit{Teichmüller geodesic} in the direction of $q$, which is parametrized by arc length with respect to the Teichmüller distance.

We now recall a formula for the Teichmüller distance, known as Kerckhoff's formula.

\begin{theorem}[Kerckhoff's formula \cite{kerckhoff1980asymptotic}]\label{theorem:KerckhoffFormula}
For any $X, Y \in \Teich(\Sigma)$, we have
\[
e^{-2 d_{\Teich}(X,Y)}
= \min_{\gamma \in \mathcal{MF}(\Sigma)}
\frac{\Ext_Y(\gamma)}{\Ext_X(\gamma)},
\]
where the minimum is attained by the vertical measured foliation of the holomorphic
quadratic differential associated with the Teichmüller map between $X$ and $Y$.
\end{theorem}

\subsection{Length of curves}

For any $X \in \Teich(\Sigma)$ and any $\gamma \in \mathcal{C}(\Sigma)$, let $\ell_X(\gamma)$ denote the length of the geodesic representative of $\gamma$ with respect to the hyperbolic metric in the conformal class determined by $X$.
Bers proved the following theorem.

\begin{theorem}[\cite{bers1985inequality}]\label{theorem:BersConstant}
    For every $X \in \Teich(\Sigma)$, there exists a pants decomposition $P \in \mathcal{P}(\Sigma)$ whose curves all have length at most $B_{g,n}$ in the hyperbolic metric associated with $X$, where $B_{g,n}$ depends only on $g$ and $n$.
\end{theorem}

The constant $B_{g,n}$ is called the \textit{Bers' constant}.
From \cref{theorem:BersConstant}, we immediately have the following corollary.

\begin{corollary}\label{corollary:BersForMulticurves}
    Let $k$ be an integer with $1 \le k \le 3g-3+n$.
    Then, for every $X \in \Teich(\Sigma)$, there exists a $k$-multicurve $\alpha = \{\alpha^1,\ldots,\alpha^k\} \in \multicurve{k}(\Sigma)$ such that $\ell_X(\alpha^i) \le B_{g,n}$ for each $\alpha^i \in \alpha$.
\end{corollary} 

As for a relation between extremal length and hyperbolic length, Maskit's inequality is well known.

\begin{theorem}[\cite{maskit1985comparison}]\label{theorem:Maskit'sLemma}
    For any $X \in \Teich(\Sigma)$ and any $\gamma \in \mathcal{C}(\Sigma)$, the    following inequality holds:
    \begin{equation*}
        \Ext_X(\gamma) \le \ell_X(\gamma)e^{\ell_X(\gamma)/2}.
    \end{equation*}
\end{theorem}

From \cref{corollary:BersForMulticurves,theorem:Maskit'sLemma}, we obtain the following corollary.

\begin{corollary}\label{corollary:UniformBoundForExtremalLength}
    Let $k$ be an integer with $1\leq k\leq 3g-3+n$.
    For every $X\in \Teich(\Sigma)$, there exists a $k$-multicurve $\alpha\in\multicurve{k}(\Sigma)$ with $\Ext_X(\alpha)\leq \delta_0$,
    where $\delta_0$ is a constant depending only on $g$ and $n$.
\end{corollary}

\begin{proof}
    For every $X\in\Teich(\Sigma)$, using \cref{corollary:BersForMulticurves}, we take a $k$-multicurve $\alpha=\{\alpha^1,\ldots,\alpha^k\} \in \multicurve{k}(\Sigma)$ with $\ell_X(\alpha^i)\leq B_{g,n}$ for each $\alpha^i\in\alpha$.
    By definition of extremal length, we can take a conformal metric $\rho$ with 
    \begin{equation*}
        \Ext_X(\alpha)\leq \frac{\ell_\rho(\alpha)^2}{\Area(\rho)}+1.
    \end{equation*}
    For this $\rho$, we have 
    \begin{align*}
        \frac{\ell_\rho(\alpha)^2}{\Area(\rho)}&=\frac{\left(\sum_{i=1}^k\ell_\rho(\alpha^i)\right)^2}{\Area(\rho)} & \\
        &\leq k \sum_{i=1}^k\frac{\ell_\rho(\alpha^i)^2}{\Area(\rho)} & (\text{Cauchy-Schwartz})\\
        &\leq k\sum_{i=1}^k \Ext_X(\alpha^i) \\
        &\leq k \sum_{i=1}^k \ell_X(\alpha^i)e^{\ell_X(\alpha^i)/2} & (\text{\Cref{theorem:Maskit'sLemma}})\\
        &\leq (3g-3+n)^2B_{g,n}e^{B_{g,n}/2}.
    \end{align*}
\end{proof}

We will use the following inequality later.

\begin{lemma}[Minsky's inequality \cite{minsky1993ends}]\label{lemma:MinskyInequality}
    Let $\gamma_1,\gamma_2\in\mathcal{MF}(\Sigma)$ be measured foliations on $\Sigma$, and $X\in \Teich(\Sigma)$. Then,
    \begin{equation*}
        i(\gamma_1,\gamma_2)^2\leq \Ext_X(\gamma_1)\Ext_X(\gamma_2).
    \end{equation*}
\end{lemma}

\section{Inequality between intersection numbers and distance}

\subsection{Inequality for pants graph distance}

Hempel \cite{hempel2001three-manifolds} showed that the distance $d_{\mathcal{C}}(\gamma_1,\gamma_2)$ between any two curves $\gamma_1,\gamma_2$ in the curve graph is bounded above by a logarithmic function of their geometric intersection number $i(\gamma_1,\gamma_2)$.
This result is often referred to as Hempel's lemma.
To prove Hempel's lemma, he used a surgery argument for curves.
However, this surgery argument does not apply to the analogous problem in the $k$-multicurve graph, in particular, in the pants graph.
On the other hand, Lackenby and Yazdi \cite{lackenby2024bounds} established an inequality showing that the distance $d_{\mathcal{P}}(P_1,P_2)$ between two pants decompositions $P_1,P_2 \in \mathcal{P}(\Sigma)$ is bounded above by a logarithmic function of the intersection number $i(P_1,P_2)$.

\begin{theorem}[{\cite[Theorem 1.1]{lackenby2024bounds}}]\label{theorem:HempelLemmaForPantsGraph}
    Let $P_1,P_2\in \mathcal{P}(\Sigma)$ be pants decompositions on $\Sigma$.
    Then there exists a constant $C>0$ such that 
    \begin{equation*}
        d_{\mathcal{P}}(P_1,P_2)\leq C (1+\log i(P_1,P_2)), 
    \end{equation*}
    where $C$ depends only on $\chi(\Sigma)$.
    Moreover $C=O(|\chi(\Sigma)|^2) $ as $|\chi(\Sigma)|\to \infty$.
\end{theorem}

Furthermore, as an intermediate step in the proof of \cref{theorem:HempelLemmaForPantsGraph}, Lackenby and Yazdi derived the following estimate, whose upper bound is a quadratic function of the intersection number, but which is more explicit and simpler than the bound in \Cref{theorem:HempelLemmaForPantsGraph}.

\begin{theorem}[{\cite[Corollary 3.16]{lackenby2024bounds}}]\label{theorem:QuadraticBound}
Let $P_1, P_2$ be pants decompositions of $\Sigma$.
Then
\begin{equation}\label{eq:QuadraticEstimate}
d_{\mathcal{P}}(P_1,P_2) \leq 6 i(P_1,P_2)^2.
\end{equation}
\end{theorem}

\begin{remark}
The right-hand side of \cref{eq:QuadraticEstimate} appears implicitly in the arguments in Section~3 of \cite{lackenby2024bounds}.
In \Cref{section:Appendix}, we discuss the explicit upper bound.
\end{remark}

\subsection{Inequality of $k$-multicurve graph distance}

\begin{lemma}\label{lemma:Extension}
    Let $k$ and $l$ be integers with $1 \le k,l \le 3g-4+n$.
    For any $k$-multicurve $\alpha \in \multicurve{k}(\Sigma)$ and any $l$-multicurve $\beta \in \multicurve{l}(\Sigma)$, there exists a $(k+1)$-multicurve $\tilde{\alpha} \in \multicurve{k+1}(\Sigma)$ such that $\alpha \subset \tilde{\alpha}$ and
    \begin{equation*}
        i(\tilde{\alpha}, \beta) \le 2\, i(\alpha, \beta).
    \end{equation*}
\end{lemma}

\begin{proof}
    Choose representatives $a$ and $b$ of the multicurves $\alpha$ and $\beta$ that realize minimal intersection number in their homotopy classes, and regard them as subsets of $\Sigma$.
    Since $k \le 3g-4+b$, there exists a connected component $X$ of $\Sigma \setminus a$ that is not a pair of pants.
    Consequently, each connected component of $X\cap b$ is a subarc of the curve $b$ or a closed curve which is a component of $b$.
    Let $c_1,\ldots,c_m \subset b$ be all connected components of $X\cap b$.
    Then the number $m$ of connected components of $X\cap b$ is at most the geometric intersection number $i(\alpha,\beta)$.

    We define a new simple closed curve $a' \subset X$ as follows.
    We set $X'=X\setminus \bigcup_{i=1}^m c_i$.

    (i) If there is a simple closed curve among $c_1,\ldots,c_m$, we let $a'$ be that curve.

    (ii) If there is no simple closed curve among $\{c_1,\ldots,c_m\}$ and the subsurface $X'$ contains a simple closed curve which is essential in that subsurface, we let $a'$ be that curve.

    (iii) Finally, we consider the case where there is no simple closed curve among ${c_1,\ldots,c_m}$ and the subsurface $X'$ contains no simple closed curve that is essential in that subsurface; that is, each connected component of $X'$ is either a pair of pants, an annulus, or a disk.
    Let $C_1$ and $C_2$ be the boundary components of $X$ on which the endpoints of the subarc $c_1$ lie (possibly $C_1 = C_2$).
    Then the isotopy classes of both $C_1$ and $C_2$ are components of $\alpha$.

    \begin{figure}
       \centering
       \begin{overpic}[]{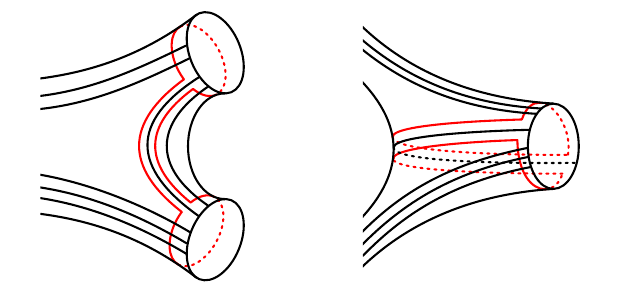}
           \put(40,40){$C_1$}
           \put(40,5){$C_2$}
           \put(20,28){\textcolor{red}{$a'$}}
           \put(92,28){$C_1=C_2$}
           \put(67,28){\textcolor{red}{$a'$}}
       \end{overpic}
       \caption{}
       \label{figure:CurveCreation}
    \end{figure}

    First, suppose that $C_1 \neq C_2$.
    The boundary of a regular neighborhood of $c_1 \cup C_1 \cup C_2$ in $X$ has a single component.
    Let $a'$ denote this boundary component (see the left part of \cref{figure:CurveCreation}).
    Then the simple closed curve $a'$ is essential in $X$.
    If $a'$ is not essential in $X$, then it is homotopic to a boundary component of $X$ or to a point.
    This implies that $X$ is either a pair of pants or an annulus, contradicting the choice of $X$.

    Next, suppose that $C_1 = C_2$, that is, the endpoints of the arc $c_1$ lie on $C_1$.
    Then the boundary of a regular neighborhood of $c_1 \cup C_1$ in $X$ has two components, say $\partial_1$ and $\partial_2$.
    At least one of $\partial_1$ and $\partial_2$ is essential in $X$.
    If neither $\partial_1$ nor $\partial_2$ is essential in $X$, then each $\partial_i$ $(i=1,2)$ is homotopic to a boundary component of $X$ or to a point.
    This implies that $X$ is either a pair of pants, an annulus, or a disk, contradicting the choice of $X$.
    In this case, we let $a'$ be an essential one of $\partial_1$ and $\partial_2$ (see the right part of \cref{figure:CurveCreation}).

    In each of the cases (i)--(iii), the intersection number $i(a',b)$ is at most $i(\alpha,\beta)$.
    We set $\tilde{\alpha} = \alpha \cup \{[a']\}$, where $[a']$ denotes the isotopy class of the simple closed curve $a'$.
    Since $a'$ is disjoint from the representative multicurve $a$ of $\alpha$, the set $\tilde{\alpha}$ is a $(k+1)$-multicurve.
    Therefore, we have
    \begin{equation*}
        i(\tilde{\alpha},\beta)
        = i(\alpha,\beta) + i([a'],\beta)
        \le 2\, i(\alpha,\beta),
    \end{equation*}
    since $i([a'],\beta) \le i(a',b) \le i(\alpha,\beta)$.
\end{proof}

By consecutively using \cref{lemma:Extension}, we obtain the following.

\begin{proposition}\label{proposition:IntersectionlessExtension}
    Let $k$ be an integer with $1\leq k \leq 3g-4+n$.
    For any two $k$-multicurves $\alpha,\beta\in\multicurve{k}(\Sigma)$, there exist pants decompositions $\tilde{\alpha}, \tilde{\beta} \in \mathcal{P}(\Sigma)$ such that $\alpha\subset \tilde{\alpha}$, $\beta\subset \tilde{\beta}$ and 
    \begin{equation*}
        i(\tilde{\alpha},\tilde{\beta}) \leq 4^{3g-3+n-k} i(\alpha,\beta).
    \end{equation*}
\end{proposition}
By the following lemma, the pants graph distance gives an upper bound for the $k$-multicurve graph distance.

\begin{lemma}\label{lemma:UpperBoundForK-multicurveDistance}
Let $\Sigma=\Sigma_{g,n}$ be a surface of complexity $\xi(\Sigma)=3g-3+n$, and let $k$ be an integer with $1\le k\le \xi(\Sigma)-1$.
For any two $k$-multicurves $\alpha,\beta\in\multicurve{k}(\Sigma)$ and any pants decompositions
$\tilde{\alpha},\tilde{\beta}\in\mathcal{P}(\Sigma)$ with $\alpha\subset\tilde{\alpha}$ and $\beta\subset\tilde{\beta}$, we have
\begin{equation*}
    d_{\multicurve{k}}(\alpha,\beta)\leq d_{\mathcal{P}}(\tilde{\alpha},\tilde{\beta})\ +f(k),
\end{equation*}
where $f(k)=\min\{k, \xi(\Sigma)-k\}$.

\end{lemma}

\begin{proof}
Let $\tilde{\alpha}=P_0,P_1,\ldots,P_m=\tilde{\beta}$ be a geodesic segment in the pants graph $\mathcal{P}(\Sigma)$ from $\tilde{\alpha}$ to $\tilde{\beta}$, so that
$m=d_{\mathcal{P}}(\tilde{\alpha},\tilde{\beta})$.

We inductively construct a sequence of $k$-multicurves $\gamma_0,\gamma_1,\ldots,\gamma_m $ such that
$\gamma_i\subset P_i$ for every $i$.
First, we set $\gamma_0\coloneqq \alpha\subset P_0$.
Next, if $\gamma_{i-1}\subset P_i$, then we define $\gamma_i:=\gamma_{i-1}$.
We consider the case of $\gamma_{i-1}\not\subset P_i$.
Since $P_{i-1}$ and $P_i$ are adjacent in the pants graph, there exist curves $\delta_i\in P_{i-1}$ and $\delta_i'\in P_i$ such that $P_{i-1}\setminus\{\delta_i\}$ coincides with $P_i\setminus\{\delta_i'\}$.
Now $\gamma_{i-1}$ is not contained in $P_i$; therefore $\gamma_{i-1}$ must contain $\delta_i$ as a component.
Choosing any curve $c_i'\in P_i\setminus(\gamma_{i-1}\setminus\{\delta_i\})$,
we define $\gamma_i:=(\gamma_{i-1}\setminus\{\delta_i\})\cup\{c_i'\}$.
Then $\gamma_i\subset P_i$ and the two $k$-multicurves $\gamma_{i-1}$ and $\gamma_i$ differ by replacing one curve, so $d_{\multicurve{k}}(\gamma_{i-1},\gamma_i)\leq 1$ for every $i=1,\ldots,m$.

Now $\gamma_m\subset P_m=\tilde{\beta}$ and $\beta\subset \tilde{\beta}$.
For a fixed pants decomposition $P$ of $\Sigma$, any two $k$-multicurves contained in $P$ can be
connected in $\multicurve{k}(\Sigma)$ by successively replacing one curve at a time, and the number of
replacements needed is at most $\min\{k,\xi(\Sigma)-k\}=f(k)$.
Therefore, we have $d_{\multicurve{k}}(\gamma_m,\beta)\le f(k)$.
Combining these estimates and using the triangle inequality, we obtain
\begin{align*}
    d_{\multicurve{k}}(\alpha,\beta) &\leq \sum_{i=1}^m d_{\multicurve{k}}(\gamma_{i-1},\gamma_i) + d_{\multicurve{k}}(\gamma_m,\beta) \\
    & \leq m+f(k) =d_{\mathcal{P}}(\tilde{\alpha},\tilde{\beta})+f(k),
\end{align*}
as desired.
\end{proof}

\begin{corollary}\label{corollary:UpperBoundForTheK-multicurveDistance}
    Let $\Sigma=\Sigma_{g,n}$ be a surface with complexity $\xi(\Sigma)=3g-3+n$, and let $k$ be an integer with $1\le k\le \xi(\Sigma)-1$.
    For any $k$-multicurves $\alpha,\beta\in\multicurve{k}(\Sigma)$, we have 
    \begin{equation*}
        d_{\multicurve{k}}(\alpha,\beta)\leq 6\cdot 4^{6g-6+2n-2k}i(\alpha,\beta)^2+f(k),
    \end{equation*}
    where $f(k)=\min\{k,\xi(\Sigma)-k\}$.
\end{corollary}

\begin{proof}
    For any $k$-multicurves $\alpha,\beta\in\multicurve{k}(\Sigma)$, let $\tilde{\alpha},\tilde{\beta}\in \mathcal{P}(\Sigma)$ be the pants decompositions in \cref{proposition:IntersectionlessExtension}.
    Then we have 
    \begin{align*}
        d_{\multicurve{k}}(\alpha,\beta) &\leq d_{\mathcal{P}}(\tilde{\alpha},\tilde{\beta})+f(k) & \text{by \cref{lemma:UpperBoundForK-multicurveDistance}}\\
        & \leq 6i(\tilde{\alpha},\tilde{\beta})^2+f(k) & \text{by \cref{theorem:QuadraticBound}} \\
        & \leq 6\cdot 4^{6g-6+2n-2k}i(\alpha,\beta)^2+f(k) & \text{by \cref{proposition:IntersectionlessExtension}.}
    \end{align*}
\end{proof}

\begin{remark}
In the above proof, if we use \Cref{theorem:HempelLemmaForPantsGraph} instead of \Cref{theorem:QuadraticBound}, we obtain an upper bound on the distance in the $k$-multicurve graph given by a logarithmic function of the intersection number.
\end{remark}

\section{Quasi-isometry between $\elTeich^k$ and $\multicurve{k}$}

Fix a sufficiently small constant $\epsilon_0>0$ so that the collar lemma holds, and let $k$ be an integer with $1 \le k \le 3g-3+n$.
For any $k$-multicurve $\alpha = \{\alpha^1,\ldots,\alpha^k\} \in \multicurve{k}(\Sigma)$,
we define a subset $\Thin_\alpha \subset \Teich(\Sigma)$ by
\begin{equation*}
    \Thin_\alpha
    = \bigl\{ X \in \Teich(\Sigma) \mid \Ext_X(\alpha^i) \le \epsilon_0
    \text{ for each } \alpha^i \in \alpha \bigr\}.
\end{equation*}

Let $(\elTeich^k(\Sigma), d_e)$ denote the metric space obtained by electrifying the metric space $(\Teich(\Sigma),d_{\Teich})$ along the family $\{\Thin_\alpha\}_{\alpha \in \multicurve{k}(\Sigma)}$.
For each $\alpha \in \multicurve{k}(\Sigma)$, let $v_\alpha \in \elTeich^k(\Sigma)$
denote the vertex corresponding to $\Thin_\alpha$ that is added in the electrification.
We define a map
\begin{equation*}
    I \colon \multicurve{k}(\Sigma) \longrightarrow \elTeich^k(\Sigma)
\end{equation*}
by $I(\alpha) = v_\alpha$.
The purpose of this section is to prove the following theorem.

\begin{theorem}\label{theorem:K-ElectricTeichmullerSpaceAndK-multicurveGraphAreQI}
    The map $I \colon \multicurve{k}(\Sigma) \to \elTeich^k(\Sigma)$ is a   quasi-isometry.
\end{theorem}

Using the upper bound on the distance in the $k$-multicurve graph in terms of the intersection number (\cref{corollary:UpperBoundForTheK-multicurveDistance}), we prove \cref{theorem:K-ElectricTeichmullerSpaceAndK-multicurveGraphAreQI} in essentially the same way as in \cite{masur1999complex}.
For the sake of completeness, we include the proof below.

\begin{lemma}\label{lemma:QuasiDensityOfI}
    The map $I$ is $(\frac12 \log \frac{\delta_0}{\epsilon_0} + \frac12)$-quasi-dense, where the constant $\delta_0$ depends only on $\Sigma$, as in \Cref{corollary:UniformBoundForExtremalLength}.
\end{lemma}

\begin{proof}
    By \Cref{corollary:UniformBoundForExtremalLength}, for each $X \in \Teich(\Sigma)$ there exists a $k$-multicurve $\alpha = \{\alpha^1,\ldots,\alpha^k\}$ such that $\Ext_X(\alpha) \le \delta_0$.
    Let $q_\alpha$ be the holomorphic quadratic differential on $X$ whose horizontal measured foliation is $\alpha$.
    Let $\{X_t\}_{t \ge 0}$ be the Teichmüller geodesic from $X$ in the direction of $-q_\alpha$, where $X_t \coloneqq X_{-q_\alpha,e^{2t}}$.
    This Teichmüller geodesic $\{X_t\}_{t \ge 0}$ contracts the horizontal direction of $q_\alpha$.
    Then, by \Cref{theorem:KerckhoffFormula}, for every $t \ge 0$,
    \begin{equation*}
        \frac{\Ext_{X_t}(\alpha)}{\Ext_X(\alpha)} = e^{-2t}.
    \end{equation*}
    If $t = \tfrac12 \log \tfrac{\delta_0}{\epsilon_0}$, then
    \begin{equation*}
        \Ext_{X_t}(\alpha^i) \leq \Ext_{X_t}(\alpha) = \frac{\epsilon_0}{\delta_0}\Ext_X(\alpha) \leq \epsilon_0
    \end{equation*}
    for each $\alpha^i \in \alpha$.
    Therefore, we have $X_t \in \Thin_\alpha$, and hence
    \begin{equation*}
        d_e(X, v_\alpha)
        \leq \frac12 \log \frac{\delta_0}{\epsilon_0} + \frac12.
    \end{equation*}
\end{proof}

\begin{lemma}\label{lemma:LipschitzPropertyForTheMapI1}
    Suppose that $k$ is an integer with $1 \le k \le 3g-4+n$.
    Then, for any $\alpha,\beta \in \multicurve{k}(\Sigma)$, we have
    \begin{equation*}
        d_e(v_\alpha, v_\beta) \le d_{\multicurve{k}}(\alpha,\beta).
    \end{equation*}
\end{lemma}

\begin{proof}
    Assume that $d_{\multicurve{k}}(\alpha,\beta) = 1$.
    In this case, the intersection $\Thin_\alpha \cap \Thin_\beta \subset \Teich(\Sigma)$ is nonempty.
    Choose $X \in \Thin_\alpha \cap \Thin_\beta$.
    Then
    \begin{equation*}
        d_e(v_\alpha, v_\beta)
        \le d_e(v_\alpha, X) + d_e(X, v_\beta)
        = 1.
    \end{equation*}
    This proves the desired inequality.
\end{proof}

\begin{lemma}\label{lemma:LipschitzPropertyForTheMapI2}
    Let $k=3g-3+n$.
    Then there exists a constant $K_0>0$ such that, for any $\alpha,\beta\in\multicurve{k}(\Sigma)$, we have
    \begin{equation*}
        d_e(v_\alpha,v_\beta)\leq K_0 d_{\multicurve{k}}(\alpha,\beta).
    \end{equation*}
\end{lemma}

\begin{proof}
    Let $\alpha = \{\alpha^1,\ldots,\alpha^k\}$ and
    $\beta = \{\beta^1,\ldots,\beta^k\}$ be elements of
    $\multicurve{k}(\Sigma)$ with
    $d_{\multicurve{k}}(\alpha,\beta) = 1$.
    Then, after relabeling if necessary, we may assume that
    $\alpha^i = \beta^i$ for $2 \le i \le k$, and that
    \begin{itemize}
        \item $i(\alpha^1,\beta^1) = 1$ if one of the connected components of
        $\Sigma \setminus \bigcup_{i=2}^k \alpha^i$ is a once-holed torus;
        \item $i(\alpha^1,\beta^1) = 2$ if one of the connected components of
        $\Sigma \setminus \bigcup_{i=2}^k \alpha^i$ is a four-holed sphere.
    \end{itemize}
    In either case, there exists $X\in \Teich(\Sigma)$ such that 
    \begin{equation*}
        \ell_X(\alpha^1)\leq D, \ell_X(\beta^1)\leq D,\ \text{and}\ \ell_X(\alpha^i)=\ell_X(\beta^i)\leq D\ (2\leq i\leq k),
    \end{equation*}
    where $D>0$ is a universal constant.
    Therefore there exists a universal constant $D'>0$ with $\Ext_X(\alpha)\leq D' \text{ and } \Ext_X(\beta)\leq D'$.
    Let $q_\alpha,q_\beta$ be holomorphic quadratic differentials on $X$ whose horizontal measured foliations equal $\alpha, \beta$, respectively.
    Let $\{X_t\}, \{Y_t\}$ be the Teichmüller geodesics starting from $X$, where $X_t=X_{-q_\alpha,e^{2t}} $ and $ Y_t = X_{-q_\beta,e^{2t}}$.
    Setting $t=\frac12\log\frac{D'}{\epsilon_0}$, by \cref{theorem:KerckhoffFormula}, we have 
    \begin{equation*}
        \Ext_{X_t}(\alpha)\leq \epsilon_0 \text{ and } \Ext_{Y_t}(\beta)\leq \epsilon_0.
    \end{equation*}
    Thus we see that $X_t\in \Thin_\alpha$ and $Y_t\in\Thin_\beta$, and hence
    \begin{align*}
        d_e(v_\alpha,v_\beta) &\leq d_e(X_t,Y_t)+1 \\
        &\leq d_{\Teich}(X_t,Y_t)+1 \\
        &\leq d_{\Teich}(X_t,X)+d_{\Teich}(X,Y_t)+1 \\
        &\leq \log\frac{D'}{\epsilon_0}+1.
    \end{align*}
    This yields the desired conclusion.
\end{proof}

\begin{lemma}\label{lemma:QuasiDenistiyOfQuasiInverse}
    For each $X\in\Teich(\Sigma)$, let $\alpha_X$ and $\alpha'_X$ be $k$-multicurves with $\Ext_X(\alpha_X)\leq \delta_0$ and $\Ext_X(\alpha_X')\leq \delta_0$.
    Then, the distance between $\alpha_X$ and $\alpha'_X$ in $\multicurve{k}(\Sigma)$ is at most $\delta_1$, where $\delta_1$ depends only on $\delta_0$, $g$ and $n$.
\end{lemma}

\begin{proof}
    By \cref{lemma:MinskyInequality}, we obtain 
    \begin{equation*}
        i(\alpha_X,\alpha'_X)^2\leq\Ext_X(\alpha_X)\Ext_X(\alpha'_X) \leq \delta_0^2. 
    \end{equation*}
    From this and \cref{corollary:UpperBoundForTheK-multicurveDistance}, we obtain 
    \begin{equation*}
        d_{\multicurve{k}}(\alpha_X,\alpha'_X)\leq K_1i(\alpha_X,\alpha'_X)^2+K_2\leq K_1\delta_0^2+K_2\eqqcolon \delta_1.
    \end{equation*} 
\end{proof}

By \Cref{corollary:UniformBoundForExtremalLength}, for each
$X \in \Teich(\Sigma)$, there exists a $k$-multicurve
$\alpha_X$ such that $\Ext_X(\alpha_X) \le \delta_0$.
We define a map
\begin{equation*}
    \Phi \colon \elTeich^k(\Sigma) \to \multicurve{k}(\Sigma)
\end{equation*}
by setting $\Phi(X) = \alpha_X$ for $X \in \Teich(\Sigma)$ and
$\Phi(v_\alpha) = \alpha$ for the vertex $v_\alpha$ corresponding to
$\Thin_\alpha$ that is added in the electrification.
Although the choice of $\alpha_X$ is not unique, it follows from
\Cref{corollary:UniformBoundForExtremalLength} that any two such choices
differ by a distance that is uniformly bounded.

\begin{proof}[Proof of \Cref{theorem:K-ElectricTeichmullerSpaceAndK-multicurveGraphAreQI}]
    The quasi-density of the map $I$ follows from \cref{lemma:QuasiDensityOfI}.
    We check that there exist constants $A\geq 1$ and $B\geq 0$ such that, for each $\alpha,\beta\in\multicurve{k}(\Sigma)$, 
    \begin{equation}\label{eq:QuasiIsometryInequalityOfI}
        \frac1A d_e(I(\alpha),I(\beta)) -B \leq d_{\multicurve{k}}(\alpha,\beta)\leq A d_e(I(\alpha),I(\beta)) +B.
    \end{equation}
    The right-hand inequality of \cref{eq:QuasiIsometryInequalityOfI} follows from \cref{lemma:LipschitzPropertyForTheMapI1,lemma:LipschitzPropertyForTheMapI2}.

    We verify the left-hand inequality of \cref{eq:QuasiIsometryInequalityOfI} below.
    Fix $X,Y\in \Teich(\Sigma)$ with $d_{\Teich}(X,Y)\leq 1$.
    By \cref{theorem:KerckhoffFormula}, we have 
    \begin{equation*}
        \frac{\Ext_X(\Phi(Y))}{\Ext_Y(\Phi(Y))}\leq e^2.
    \end{equation*}
    Therefore we see that $\Ext_X(\Phi(Y))\leq e^2\delta_0$, and hence, by \cref{lemma:MinskyInequality},
    \begin{equation*}
        i(\Phi(X),\Phi(Y))^2\leq \Ext_X(\Phi(X))\Ext_X(\Phi(Y))\leq e^2\delta_0^2.
    \end{equation*}
    Thus we obtain 
    \begin{equation*}
        d_{\multicurve{k}}(\Phi(X),\Phi(Y))\leq K_1i(\Phi(X),\Phi(Y))^2+K_2\leq K_1e^2\delta_0^2+K_2\eqqcolon K_3.
    \end{equation*}
    For any $X,Y\in \Teich(\Sigma)$, we set $n=\lfloor d_{\Teich}(X,Y)\rfloor$.
    Choose a sequence $\{X_i\}_{i=0}^{n+1}\subset \Teich(\Sigma)$ with $X_0=X, X_{n+1}=Y$ and $d_{\Teich}(X_i,X_{i+1})\leq 1$ for each $i=0,\ldots,n$.
    Then 
    \begin{equation*}
        d_{\multicurve{k}}(\Phi(X),\Phi(Y))\leq \sum_{i=0}^n d_{\multicurve{k}}(\Phi(X_i),\Phi(X_{i+1}))\leq K_3 (n+1).
    \end{equation*}
    This shows that
    \begin{equation}\label{eq:BoundOnGraphDistanceByTeichmullerDistance}
        d_{\multicurve{k}}(\Phi(X),\Phi(Y))\leq K_3 (d_{\Teich}(X,Y)+2).
    \end{equation}

    For any $\alpha, \beta \in \multicurve{k}(\Sigma)$, we choose a path $c \subset \elTeich^k(\Sigma)$ from $v_\alpha$ to $v_\beta$ such that $\ell_{\elTeich^k}(c) \le d_e(v_\alpha, v_\beta) + 1$.
    Let
    \begin{equation*}
    v_\alpha = v_{\alpha_0}, v_{\alpha_1}, \ldots, v_{\alpha_m} = v_\beta
    \qquad
    (\alpha_i \in \multicurve{k}(\Sigma),\ 0 \le i \le m)
    \end{equation*}
    denote the added vertices of $\elTeich^k(\Sigma)$ that the path $c$ passes through,
    listed in the order in which they are encountered along $c$.
    Then, the number $m$ of the vertices is at most $d_e(v_\alpha,v_\beta)+1$.
    For each $i$ with $1 \le i \le m$, let $c_i$ denote the subpath of $c$ from $v_{\alpha_{i-1}}$ to $v_{\alpha_i}$.
    Let $X_i \in \Teich(\Sigma)$ be a point on $c_i$ closest to $v_{\alpha_{i-1}}$, and let $X_i' \in \Teich(\Sigma)$ be a point on $c_i$ that is closest to $v_{\alpha_i}$ (see \cref{figure:SettingX}).
    \begin{figure}
       \centering
       \begin{overpic}[]{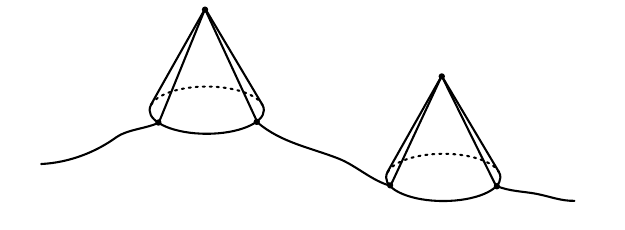}
           \put(31,37){$v_{\alpha_{i-1}}$}
           \put(69,26){$v_{\alpha_i}$}
           \put(50,9){$c_i$}
           \put(39,12){$X_i$}
           \put(60,2){$X_i'$}
           \put(77,2){$X_{i+1}$}
           \put(22,12){$X_{i-1}'$}
           \put(11,15){$c_{i-1}$}
           \put(88,6){$c_{i+1}$}
       \end{overpic}
       \caption{}
       \label{figure:SettingX}
    \end{figure}
    Then, $X_i\in \Thin_{\alpha_{i-1}}$ and $X_i'\in \Thin_{\alpha_i}$.
    In addition, we have
    \begin{equation}\label{eq:PathC_i}
        \ell_{\elTeich^k}(c_i)\geq d_{\Teich}(X_i,X_i')+1.
    \end{equation}
    Therefore, by \cref{lemma:QuasiDenistiyOfQuasiInverse} and inequalities \cref{eq:BoundOnGraphDistanceByTeichmullerDistance}, \cref{eq:PathC_i}, we have 
    \begin{align*}
        &\quad d_{\multicurve{k}}(\alpha_{i-1},\alpha_i) \\
        &\leq d_{\multicurve{k}}(\alpha_{i-1},\Phi(X_i))+d_{\multicurve{k}}(\Phi(X_i),\Phi(X_i'))+d_{\multicurve{k}}(\Phi(X'_i),\alpha_i) \\
        & \leq d_{\multicurve{k}}(\Phi(X_i),\Phi(X_i'))+2\delta_1 \\
        &\leq K_3(d_{\Teich}(X_i,X_i')+2)+2\delta_1 \\
        &\leq K_3(\ell_{\elTeich^k}(c_i)+1)+2\delta_1.
    \end{align*}
    Moreover, we have 
    \begin{align*}
        d_{\multicurve{k}}(\alpha,\beta) &\leq \sum_{i=1}^m d_{\multicurve{k}}(\alpha_{i-1},\alpha_i) \\
        &\leq K_4 m+K_3\sum_{i=1}^m \ell_{\elTeich^k}(c_i)\\
        &\leq (K_4+K_3)(d_e(v_\alpha,v_\beta)+1)
    \end{align*}
    This proves that the map $I$ is a quasi-isometry map.
\end{proof}

\appendix
\section{Quadratic upper bound for pants graph distance via intersection number}\label{section:Appendix}

\begin{theorem}[{\cite[Corollary 3.21]{lackenby2024bounds}}]\label{theorem:QuadraticUpperBoundsForPantsGraphDistance}
    Let $P,P'$ be pants decompositions of $\Sigma$.
    Then the distance $P$ and $P'$ in $\mathcal{P}(\Sigma)$ is $O(i(P,P')^2)$.
\end{theorem}

In this section, following the proof of Lackenby and Yazdi, we obtain a more explicit upper bound on the pants graph distance in terms of the intersection number, without using big O notation.

We recall the material in Section 3 of the paper \cite{lackenby2024bounds}.
Let $\Sigma=\Sigma_{g,n}$ be a surface of genus $g$ with $n$ punctures.
A \textit{pre-triangulation} of $\Sigma$ is a finite 1-complex embedded into $\Sigma$ such that each connected component of its complement is topologically either a disk, an annulus or a pair of pants.
Here, we regard a simple closed curve with no vertices as a 1-complex.
Let $\mathcal{T}$ be a pre-triangulation of $\Sigma$.
We choose an ordering of the edges of $\mathcal{T}$, denoted by $\mathrm{O}=(e_1,\ldots,e_i)$.
We construct two sequences of subsurfaces $\{N_j\}_{j=1}^i$ and $\{F_j\}_{j=1}^i$  of $\Sigma$, associated with the pre-triangulation $\mathcal{T}$ and the ordering $\mathrm{O}$ of its edges.
First, we define $N_1\subset \Sigma$ to be a regular neighborhood of $e_1$.
Next, we inductively define $N_{j+1}\subset \Sigma$ to be a regular neighborhood of $N_j\cup e_{j+1}$.
Finally, capping off every disk component of $\Sigma\setminus N_j$, we obtain a subsurface $F_j$.
For each $j$, let $\gamma_j$ denote the multicurve (namely, a set of isotopy classes of essential simple closed curves) consisting of all boundary components of $F_j$ that are essential in $\Sigma$, where parallel simple closed curves are identified \cite[Construction 3.10]{lackenby2024bounds}.
Then, the union $\bigcup_{j=1}^i \gamma_j$ gives a pants decomposition of $\Sigma$ \cite[Lemma 3.13]{lackenby2024bounds}.
We denote the pants decomposition by $\mathbf{P}(\mathcal{T},\mathrm{O})$.

If $\mathrm{O}'$ is obtained from an ordering $\mathrm{O}$ of the edges of $\mathcal{T}$ by swapping two consecutive indices, then the distance between $\mathbf{P}(\mathcal{T},\mathrm{O})$ and $\mathbf{P}(\mathcal{T},\mathrm{O}')$ is bounded above by a uniform constant $C$ \cite[Lemma 3.18]{lackenby2024bounds}. 
Now let $\mathrm{O}$ and $\mathrm{O}'$ be arbitrary orderings of the edges of $\mathcal{T}$.
Since any permutation of $i$ elements can be written as a product of at most
$1+\cdots+(i-1)=i(i-1)/2$ adjacent transpositions, $\mathrm{O}$ can be transformed into $\mathrm{O}'$ by at most $i(i-1)/2$ adjacent transpositions, where $i$ is the number of edges of $\mathcal{T}$. 
Therefore, the distance between $\mathbf{P}(\mathcal{T},\mathrm{O})$ and $\mathbf{P}(\mathcal{T},\mathrm{O}')$ is at most $Ci(i-1)/2$
\cite[Corollary 3.19]{lackenby2024bounds}.

Let $P$ and $P'$ be pants decompositions of $\Sigma$, and assume that they are in minimal position.
If $P$ and $P'$ share some components, then by cutting $\Sigma$ along the common curves and restricting to the resulting subsurfaces, we may assume without loss of generality that $P$ and $P'$ have no common component.
Let $\mathcal{T}_{P\cup P'}$ be the pre-triangulation obtained by superimposing $P$ and $P'$.
Then the number of edges of $\mathcal{T}_{P\cup P'}$ is at most $2\,i(P,P')$.
Moreover, there exist orderings $\mathrm{O}$ and $\mathrm{O}'$ of the edges of $\mathcal{T}_{P\cup P'}$ such that $\mathbf{P}(\mathcal{T}_{P\cup P'},\mathrm{O})=P$ and $\mathbf{P}(\mathcal{T}_{P\cup P'},\mathrm{O}')=P'$.
Therefore,
\begin{equation}\label{equation:QuadraticUpperBoundsForPantsGraphDistance}
    \begin{split}
        d_{\mathcal{P}}(P,P') 
        &\leq C\cdot \frac{2\,i(P,P')(2\,i(P,P')-1)}{2} \\
        &= C\bigl(2\,i(P,P')^{2}-i(P,P')\bigr).
    \end{split}
\end{equation}

Thus, finding the constant $C$ explicitly, we obtain an inequality between the pants graph distance and the intersection number.
By the following proposition, we see that the best possible constant $C$ is $3$.

\begin{proposition}
    Let $\mathcal{T}$ be a pre-triangulation of $\Sigma$.
    Let $\mathrm{O}$ and $\mathrm{O}'$ be orderings of edges of $\mathcal{T}$.
    Suppose that $\mathrm{O}'$ is obtained by swapping two consecutive indices of $\mathrm{O}$.
    Then, 
    \begin{equation*}
        d_{\mathcal{P}}(\mathbf{P}(\mathcal{T},\mathrm{O}),\mathbf{P}(\mathcal{T},\mathrm{O}'))\leq 3.
    \end{equation*}
\end{proposition}

\begin{proof}
    Following the argument in the proof of \cite[Lemma 3.18]{lackenby2024bounds}, we verify this proposition.
    Let $\mathrm{O}$ be an ordering of edges of $\mathcal{T}$ and let $\mathrm{O}'$ be the ordering obtained by swapping $j$-th and $(j+1)$-th edges.
    We write
    \begin{equation*}
        \mathrm{O}=(e_1,\ldots,e_j,e_{j+1},\ldots,e_i) \text{ and } 
        \mathrm{O}'=(e_1,\ldots,e_{j+1},e_{j},\ldots,e_i).
    \end{equation*}
    Let $\{N_k\}_{k=1}^i$ and $\{F_k\}_{k=1}^i$ be the sequences of subsurfaces associated with $\mathcal{T}$ and $\mathrm{O}$, and let $\{N'_k\}_{k=1}^i$ and $\{F'_k\}_{k=1}^i$ be those associated with $\mathcal{T}$ and $\mathrm{O}'$.
    For simplicity, we denote $\mathbf{P}(\mathcal{T},\mathrm{O})$ and $\mathbf{P}(\mathcal{T},\mathrm{O}')$ by $P$ and $P'$, respectively.
    By the construction of subsurfaces, we find that $F_k=F_k'$ if $|k-j|\geq 1$.
    Therefore, every curve in the symmetric difference $P \triangle P'$ is contained in the subsurface
    \begin{equation*}
        F_{j+1}\setminus F_{j-1}=F'_{j+1}\setminus F'_{j-1}.
    \end{equation*}

    Since $\chi(F_{j+1}\setminus F_{j-1})\geq -2$ \cite[Lemma 3.12]{lackenby2024bounds}, the number of connected components of $F_{j+1}\setminus F_{j-1}$ that are neither a disk, an annulus, or a pair of pants is at most one.
    After discarding disks and annuli in connected components of $F_{j+1}\setminus F_{j-1}$, the subsurface $F_{j+1}\setminus F_{j-1}$ is either
    \begin{enumerate}
        \item an essential four-holed sphere, an essential once-holed torus, an essential twice-holed torus or
        \item a disjoint union of at most two essential pairs of pants.
    \end{enumerate}
    If $F_{j+1}\setminus F_{j-1}$ is of type (2), the pants decompositions $P$ and $P'$ coincide.
    Thus, we may assume that $F_{j+1}\setminus F_{j-1}$ is of type (1).
    Let $N$ denote the subsurface $N_{j-1}=N'_{j-1}$.
    If at least one of $e_j$ and $e_{j+1}$ is a simple closed curve, the subsurface $F_{j+1}\setminus F_{j-1}$ is of type (2). 
    If at least one endpoint of $e_j$ or $e_{j+1}$ is disjoint from $N$, then $\partial N_j$ and $\partial N'_j$ can be isotoped to be disjoint.
    Therefore, we may assume that every endpoint of $e_j$ and $e_{j+1}$ lies in $N$.
    In this case, the ways in which $e_j$ and $e_{j+1}$ attach to $N$ are shown in \Cref{figure:HowToAttach}.
    \begin{figure}
       \centering
       \begin{overpic}[]{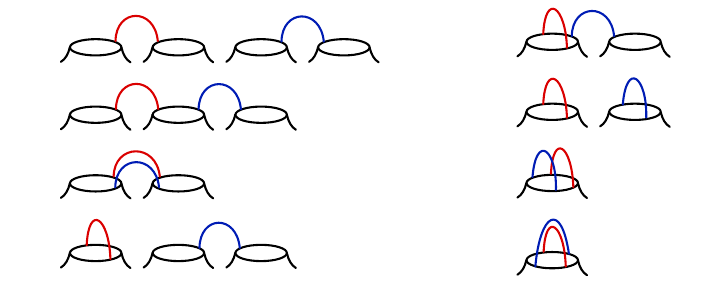}
           \put(2,34){(a)}
           \put(2,24.5){(b)}
           \put(2,15){(c)}
           \put(2,5){(d)}
           \put(66.5,34.5){(e)}
           \put(66.5,25){(f)}
           \put(66.5,15){(g)}
           \put(66.5,4.5){(h)}
       \end{overpic}
       \caption{The red and blue arcs denote $e_j$ and $e_{j+1}$ connecting boundary components of $N$.}
       \label{figure:HowToAttach}
    \end{figure}
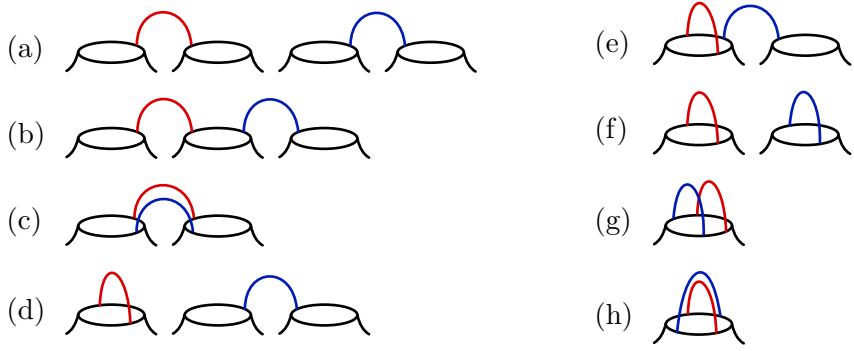

    In the cases (a), (d) and (f) in \cref{figure:HowToAttach}, the subsurface $F_{j+1}\setminus F_{j-1}$ is a disjoint union of at most two essential pairs of pants, after discarding disk and annulus components of $F_{j+1}\setminus F_{j-1}$. 
    Then, there is no difference between $P$ and $P'$.

    In the cases (b), (c), (e) and (g) in \cref{figure:HowToAttach}, one of the connected components of the subsurface $F_{j+1}\setminus F_{j-1}$ can be a four-holed sphere $S$.
    Then the curves $\partial N_j$ and $\partial N'_j$ intersect at most four times
    (see \cref{figure:IntersectingCurve}, note that now the endpoints of $e_j$ and $e_{j+1}$ may lie in a different component of $\partial N$, and the two bands are not linked).
    In $S$, the unique curve of $P$ intersects the unique curve of $P'$ at most four times.
    Since $P$ and $P'$ coincide outside the four-holed sphere $S$, the distance $d_{\mathcal{P}}(P,P')$ is at most the maximum distance $d_1$ in $\mathcal{C}(\Sigma_{0,4})$ between two curves intersecting at most four times.
    If two curves on a four-holed sphere intersect at most four times, then, up to homeomorphism, they are one of two types shown in \cref{figure:4TimesHoledSphere}.
    Therefore, the maximum distance $d_1$ is at most $2$, and hence, so is $d_{\mathcal{P}}(P,P')$.
    \begin{figure}
       \centering
       \begin{overpic}[scale=0.8]{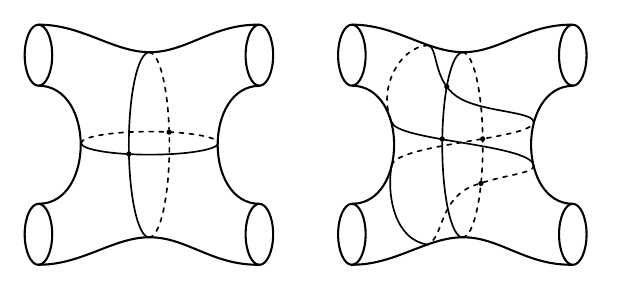}
       \end{overpic}
       \caption{}
       \label{figure:4TimesHoledSphere}
    \end{figure}

    We consider the case (h).
    In this case, one of the connected components of $F_{j+1}\setminus F_{j-1}$ is a once-holed torus or a twice-holed torus.
    We denote the connected component by $S$.
    Then $P$ and $P'$ coincide outside $S$, and they differ only by curves contained in $S$.
    Each of $\partial N_j\cap S$ and $\partial N'_j\cap S$ consists of two curves (see \cref{figure:IntersectingCurve}).
    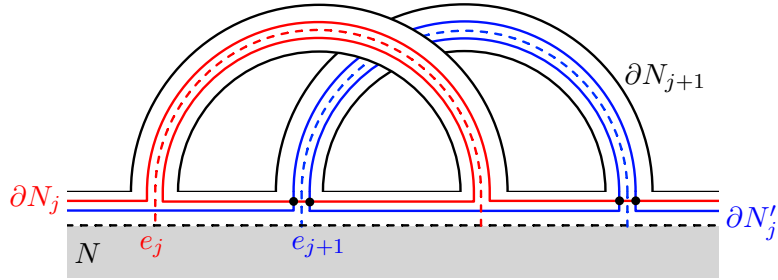
\begin{figure}
       \centering
       \begin{overpic}[scale=0.9]{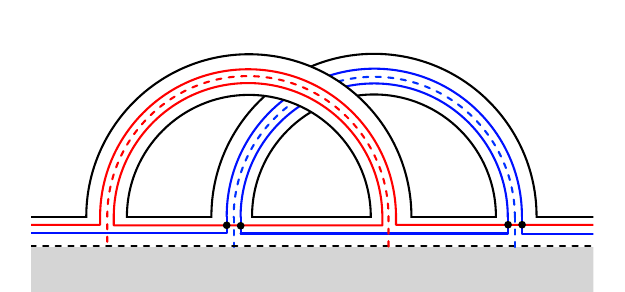}
           \put(6,2){$N$}
           \put(-3,10){\textcolor{red}{$\partial N_j$}}
           \put(96,7){\textcolor{blue}{$\partial N'_j$}}
           \put(82,27){$\partial N_{j+1}$}
           \put(15,4){\textcolor{red}{$e_j$}}
           \put(36,4){\textcolor{blue}{$e_{j+1}$}}
       \end{overpic}
       \caption{The white region (bounded by $\partial N_{j+1}$ and the black dotted line) is $S$.}
       \label{figure:IntersectingCurve}
    \end{figure}
    We denote the components of $\partial N_j\cap S$ by $c_1,c_2$ and those of $\partial N'_j\cap S$ by $c'_1,c'_2$.
    From \cref{figure:IntersectingCurve}, we find the following.
    \begin{claim}\label{claim:ClaimsFromTheFigure}
        \begin{enumerate}
            \item Each connected component of $\partial N_j\cap S$ intersects each connected component of $\partial N'_j\cap S$ at most once.
            \item The curves $c_1, c_2, c'_1$ and $c'_2$ are all non-separating in $S$.
        \end{enumerate}
    \end{claim}
    
    First, we suppose that $S$ is a once-holed torus.
    Then $c_1$ and $c_2$ are isotopic, and the same holds for $c'_1$ and $c'_2$.
    Since the distance in the pants graph between $P$ and $P'$ is bounded above by the distance $d_2$ in the curve graph of $S$ between $\gamma\coloneqq [c_1]=[c_2]$ and $\gamma'=[c'_1]=[c'_2]$.
    By \Cref{claim:ClaimsFromTheFigure} (1), it follows that the geometric intersection number $i(\gamma,\gamma')$ equals $1$. 
    Thus, we have $d_2=1$.

    Next, we suppose that $S$ is a twice-holed torus.
    Then each of the multicurves $\gamma=\{[c_1],[c_2]\}$ and $\gamma'=\{[c'_1],[c'_2]\}$ is a pants decomposition of $S$.
    By \cref{claim:ClaimsFromTheFigure}, up to the action of the mapping class group, $\gamma$ and $\gamma'$ are as shown in \cref{figure:TwiceHoledTorus}.
    \begin{figure}
       \centering
       \begin{overpic}[]{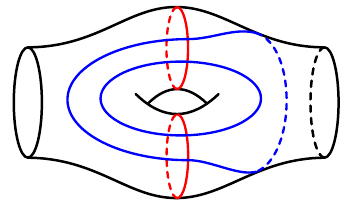}
       \end{overpic}
       \caption{The red multicurve is $\gamma$ and the blue multicurve is $\gamma'$. 
       The distance between the two multicurves is $3$.}
       \label{figure:TwiceHoledTorus}
    \end{figure}
    Now the distance in the pants graph $\mathcal{P}(\Sigma)$ between $P$ and $P'$ is bounded above by the distance $d_3$ in the pants graph of $S$ between $\gamma$ and $\gamma'$, and we find that $d_3=3$.

    Thus, in all cases, the distance in $\mathcal{P}(\Sigma)$ between $P=\mathbf{P}(\mathcal{T},\mathrm{O})$ and $P'=\mathbf{P}(\mathcal{T},\mathrm{O}')$ is at most $3$.
\end{proof}

Therefore, using \Cref{equation:QuadraticUpperBoundsForPantsGraphDistance}, we may restate \Cref{theorem:QuadraticUpperBoundsForPantsGraphDistance} in the following explicit form.
\begin{theorem}
    For any two pants decompositions $P,P'\in \mathcal{P}(\Sigma)$, 
    \begin{equation*}
        d_{\mathcal{P}} (P,P')\leq 6i(P,P')^2-3i(P,P').
    \end{equation*}
\end{theorem}

\bibliographystyle{alpha}
\bibliography{ElectricTeichmullerAndMulticurveGraph.bib}

\end{document}